\documentclass{amsart}

\makeatletter
\def\eqalign#1{\null\,\vcenter{\openup\jot\m@th
  \ialign{\strut\hfil$\displaystyle{##}$&$\displaystyle{{}##}$\hfil
      \crcr#1\crcr}}\,}
\makeatother 

\newcommand\meshing[1][-1pt]{%
  \lineskiplimit=#1\relax
  \normallineskiplimit=#1\relax
  \lineskip=#1\relax
  \normallineskip=#1\relax
  \advance\jot-#1\relax 
  \advance\jot 1pt 
}

\usepackage{amssymb}
\usepackage{tikz-cd}

\newtheorem{theorem}{Theorem}[section] 
\newtheorem{lemma}[theorem]{Lemma}

\newtheorem{example}[theorem]{Example}

\newtheorem{definition}[theorem]{Definition}

\def\abs#1{|#1|}
\def\qandq{\quad\textrm{and}\quad}

\def\Smax{S_{\rm max}}
\def\artin#1#2{\bigl( {#1 \over #2} \bigr)}
\def\Z{\mathbf Z}
\def\Q{\mathbf Q}
\def\C{\mathbf C}
\def\isar{\ \smash{\mathop{\longrightarrow}\limits^{\thicksim}}\ }

\def\Kbar{\overline K}
\def\cyclK{K_p^{\rm cycl}}
\def\antiK{K_p^{\rm anti}}
\def\antiKt{K_3^{\rm anti}}
\DeclareMathOperator\Gal{Gal}
\DeclareMathOperator\disc{disc}
\DeclareMathOperator\Pic{Pic}
\DeclareMathOperator\image{Im}
\def\gp{\mathfrak{p}}
\def\gf{\mathfrak{f}}
\def\OK{\mathcal{O}}
\def\O{\mathcal{O}}

\title{Explicit computations in Iwasawa theory} 

\author{Reinier Br\"oker}
\address{Center for Communications Research\\ 
Princeton, NJ 08540\\
United States}
\email{rmbroke@idaccr.org}

\author{David Hubbard}
\address{35 Holt Circle\\
  Hamilton, NJ 08619\\
United States}
\email{dhubbard@erols.com}

\author{Lawrence C. Washington}
\address{University of Maryland\\ 
College Park, MD 20742\\
United States}
\email{lcw@math.umd.edu}


\begin{document}

\begin{abstract}
We give two algorithms to compute layers of the anticyclotomic $\Z_3$-extension
of an imaginary quadratic field. The first is based on complex multiplication
techniques for nonmaximal orders; the second is based on Kummer theory. As
an illustration of our results, we
use the mirroring principle to derive results on the structure of class groups
of nonmaximal orders.
\end{abstract}
\maketitle

\section{Introduction} \label{intro}

Let $K$ be an imaginary quadratic field, with fixed algebraic closure
$\Kbar$, and for a fixed odd prime $p$, let $K^p \subset \Kbar$ be 
the compositum of all $\Z_p$-extensions. The Galois group of $K^p/K$ 
is isomorphic to $\Z_p^2$,
and there are two ``natural'' $\Z_p$-extensions of $K$ inside $K^p$. The
{\it cyclotomic $\Z_p$-extension\/} $\cyclK$ is the $p$-part of the
extension $\bigcup_{n \geq 1} K(\zeta_{p^n})\subset\Kbar$. 
The extension $\cyclK/\Q$ is procyclic. The {\it anticyclotomic 
$\Z_p$-extension\/} $\antiK$ is implicitly defined by the property that $\antiK
\subset\Kbar$ is the unique $\Z_p$-extension of $K$ that is {\it prodihedral\/} over $\Q$,
meaning that we have\meshing
$$
\Gal(\antiK/\Q) \cong \Z_p \rtimes \Z/2\Z,
$$
where the generator of $\Gal(K/\Q) \cong \Z/2\Z$ acts by inversion on $\Z_p$.

The fields $\cyclK$ and $\antiK$ are linearly disjoint over $K$, and their
compositum equals $K^p$. Since both have Galois group $\Z_p$, both extensions
are unramified outside of $p$ by \cite[Prop.\ 13.2]{lcw}.  This article focuses on
{\it explicitly computing\/} layers of $\antiKt$ for the case where $3$ is
ramified in $K$. By computing, we mean that on input
of a positive integer $k$, we want to compute an irreducible 
polynomial $f \in K[x]$ of degree $3^k$ with\meshing
$$
K_{k} = K[x]/(f(x)) \subset \antiKt.
$$
The Galois group $\Gal(K_{k}/K)$ is cyclic of order $3^k$.

Although we believe that most of our techniques can be generalized 
to arbitrary $p$
and arbitrary splitting behavior of $p$, our restrictions
to $p=3$ and to the case that $3$ ramifies in $K$ allow us to highlight the
technical considerations that arise in those cases. Furthermore, we can use
the {\it mirror principle\/}, see Section \ref{mirror}, to obtain a criterion 
for when the $3$-parts of certain class groups are cyclic.

The main result of this
paper is that we have explicit algorithms to compute $K_k$. We use
complex multiplication (CM) techniques in Sections \ref{ringclassfield} and
\ref{subfield}, and Kummer techniques in \ref{kummersec}. The CM technique
works for any $K$; the Kummer technique is more restricted.

Previous attempts to compute initial layers of anticyclotomic $\Z_p$-extensions
of an imaginary quadratic field include \cite{brink, CK,kimoh, VH}. These papers use
a mix of class field theory and decomposition laws of primes. 

Perhaps not surprisingly, the {\it run times\/} of our algorithms are inherently 
exponential. Not only are the outputs of the algorithms polynomials of degree
$3^k$, but the CM approach computes, as intermediate step, a polynomial
whose degree and logarithmic height of its coefficients are both
$\widetilde O(|\disc(K)|^{1/2} 3^k)$. For the Kummer approach, we need a 
polynomial of degree $O(3^{k})$ over an auxilliary extension of 
degree $O(3^k)$; furthermore,
the coefficients are themselves symmetric expressions in $O(3^k |\disc(K)|)$ terms. 

Both approaches have their merits. Indeed, whereas the CM method requires the
full class group of $K$ as intermediate step, the Kummer method only looks at
the prime 3. If the class group is large, then the Kummer method is 
better for small $n$. However, the Kummer method requires working over auxilliary
extensions and this makes the method slower for larger $n$. 

We detail various techniques
we can use to reduce the size of the generating polynomial for $K_k$ in 
Section \ref{shim}. We illustrate our techniques with a variety of examples.
All examples were done using the computer algebra package MAGMA \cite{magma}
and the CM software package \cite{AECM}.

\section{Anticyclotomic extension and ring class fields} \label{ringclassfield}

Throughout this section, let $K = \Q(\sqrt{D})$ be a fixed imaginary 
quadratic field of discriminant $D$ in which $3$ is ramified. We let
$\OK$ be the maximal order of $K$. For any integer $k \geq 1$, the $k$-th layer $K_k$ of the anticyclotomic
$\Z_3$-extension of $K$ is a generalized dihedral extension of $\Q$. Hence, by
Bruckner's result (see \cite{bruckner} or \cite[Thm. 9.18]{Cox}), we know that $K_k$ 
is contained in a {\it ring class field\/}
for $K$. Since $K_k$ is unramified outside $3$, it follows that $K_k$ is contained in
a ring class field for an order $\O_N = \Z + 3^N \OK$ of index $3^N$ for 
some $N \geq 1$. 

In order to bound the exponent, we analyze ring class fields. We let $H_N$ be the
ring class field for the order $\O_N$. With this notation, $H_0$ is the Hilbert
class field of $K$. The extension $H_N/K$ is abelian and unramified outside $3$. 
The Artin map gives an isomorphism $\Pic(\O_N) \isar \Gal(H_N/K)$, with $\Pic(\O_N)$ 
the {\it Picard group\/} of $\O_N$. We have a natural exact sequence
$$
1 \rightarrow (\OK/3^N\OK)^*/\image(\OK^*)(\Z/3^N\Z)^* \rightarrow \Pic(\O_N) \rightarrow \Pic(\OK) \rightarrow 1,
$$
where the last map is given by $[I] \mapsto [I\cdot \OK]$. The kernel of the map
$\Pic(\O_N) \rightarrow \Pic(\OK)$ is naturally isomorphic to $\Gal(H_N/H_0)$; the
following lemma gives the structure of this group.

\begin{lemma}\label{galgroup} With the notation from the previous paragraph, we have
$$
\Gal(H_N/H_0) \cong
\left\{
\begin{array}{ll}
\Z/3^{N-1}\Z & \textrm{if\ } D = -3\\
\Z/3\Z \times \Z/3^{N-1}\Z & \textrm{if\ } D \not = -3, D \equiv -3 \bmod 9\\
\Z/3^N\Z & \textrm{if\ } D \equiv 3 \bmod 9
\end{array}
\right.
$$
for $N \geq 1$.
\end{lemma}

\begin{proof} \meshing
Let $\gp \mid (3)$ be the ideal of norm $3$ in $\OK$. 
We have $(\OK/\gp^{2N} )^* \cong (A/\gp^{2N})^*$, where $A$ denotes the
completion of $\OK$ at $\gp$. The ring $A$ is a tamely ramified 
quadratic extension of $\Z_3$, and it well-known that there are only
{\it two\/} such rings up to isomorphism.
For $D \equiv -3 \bmod 9$, we have $A = \Z_3[\sqrt{-3}] = \Z_3[\zeta_3]$,
and $A = \Z_3[\sqrt{3}]$ for $D \equiv 3 \bmod 9$. We analyse both cases
separately.

The unit group of $A = \Z_3[\zeta_3]$ equals 
$$
A^* = \langle -\zeta_3 \rangle \times (1 + \gp^2),
$$
and $1+\gp^2$ is torsion free. Hence, $1+\gp^2$ is a free $\Z_3$-module of rank 2.
We get
$$
(A/3^NA)^* \cong \langle -\zeta_3 \rangle \times (1+\gp^2)/(1+\gp^{2N})
\cong \Z/6\Z \times (\Z/3^{N-1}\Z)^2,
$$
and hence
$$
(A/3^NA)^*/(\Z/3^N\Z)^* \cong \Z/3\Z \times \Z/3^{N-1}\Z.
$$

For $A = \Z_3[\sqrt{3}]$, we have
$$
A^* = \langle -1 \rangle \times (1+\gp),
$$
and since $\zeta_3$ is not contained in $A$, the $\Z_3$-module $1+\gp$ is
torsion free and hence a free rank 2 module. By iteratively applying the
``cubing isomorphism'' $1+\gp^k \isar 1+\gp^{k+2}$ we see that
$$
(A/3^NA)^* \cong \langle -1 \rangle \times (1+\gp)/(1+\gp^{2N})
\cong \Z/2\Z \times \Z/3^N\Z \times \Z/3^{N-1}\Z
$$
holds. Since the module $1+\gp$ is generated over $\Z_3$ by $1+3$ and 
$1+\sqrt{3}$, we get
$$
(A/3^NA)^* /(\Z/3^N\Z)^* \cong \Z/3^N\Z.
$$

We have $\OK^* = \{ \pm 1\}$ for $D<-3$, and 
the only case where the local cube root of unity exists globally 
is $D=-3$. Quotienting by $\image(\OK^*)$ gives the lemma.
\end{proof}

For $D \equiv -3 \bmod 9$ with $D<-3$, we let $\alpha_N \in \OK$ be an element that
is congruent to $\zeta_3 \in A$ modulo $3^{N}$. (As in the proof of
Lemma \ref{galgroup}, $A$ denotes the completion of $\OK$ at $\gp$.)
This element $\alpha_N$ determines 
an {\it Artin symbol\/} $\artin{\alpha_N}{H_{N}/H_0} \in \Gal(H_{N}/H_0)$.
We let $H_{N}'$ be the fixed field of the 
order 3 subgroup $\langle \artin{\alpha_N}{H_{N}/H_0}\rangle$ and put
$$
H_\infty = \left\{
\begin{array}{ll}
\bigcup_{N \geq 1} H_N'/H_0 & \textrm{for\ } D \equiv -3 \bmod 9 \textrm{\ and\ }
D \not = -3 \\
\bigcup_{N \geq 1} H_N/H_0 & \textrm{otherwise.}
\end{array}
\right.
$$

\begin{theorem}\label{expbound}
Let $K_k$ be the $k$-th layer of the anticyclotomic $\Z_3$-extension of $K$. Then
$K_k$ is contained in the ring class field for the order $\O_{k+1} = \Z+3^{k+1}\OK$ of
index $3^{k+1}$.
\end{theorem}
\begin{proof} It is clear that $H_N/\Q$ is generalized dihedral. From the relation
$$
\O_N \subseteq \O_M \Longrightarrow H_M \subseteq H_N
$$
from class field theory, also known as the {\it Anordnungssatz\/} for
ring class fields, and Lemma \ref{galgroup} we see that 
$$
\Gal(H_\infty/H_0) \cong \Z_3.
$$
An inspection of the sizes in Lemma \ref{galgroup} now gives that the 
compositum $K_k H_0$ is contained in $H_{k+1}$. The theorem follows.
\end{proof}

The theory of {\it complex multiplication\/} provides us with a means of 
explicitly computing the extension $H_N/K$. This theory is usually only
developed for {\it maximal\/} orders, but it generalizes to nonmaximal
orders without too much difficulty. Indeed, by \cite[Thm. 11.1]{Cox}
we know that
$$
H_N = K[x]/(f_N(x)),
$$
with $f_N \in \Z[x]$ the minimal polynomial of the $j$-invariant of the complex
elliptic curve $\C/\O_N$. There are various algorithms to compute $f_N$;
we refer to \cite{ants8} and the references therein for an overview. However,
since the proven upper bound $\widetilde O(|\disc(\O_N)^2|)$ (see e.g.
\cite[Section 5]{ants8}) on the bit size of $f_N$
is believed to be the actual size of $f_N$, these algorithms are inherently
exponential. We will give various practical improvements in Section \ref{shim} to
this basic approach.  \meshing

\section{Selecting the right subfield}\label{subfield}

As before, let $K$ be a fixed imaginary quadratic field in which $3$ is ramified. We
have seen that the $k$-th layer $K_k$ of the anticyclotomic $\Z_3$-extension of $K$
is contained in the ring class field $H_{k+1}$. In this section we explain a method
to compute $K_k$ as a subfield of $H_{k+1}$. To keep the sizes of the
generating polynomials small, the examples given in this section already
use the algorithmic improvements explained in Section \ref{shim}. Magma
code to compute the examples is available at \cite{tobefilledin}.

We first assume that $K$ has trivial 3-Hilbert class field. In this case, we 
have
$$
[H_{k}:K_k] = \#\Pic(\OK) \quad \textrm{for\ } D \equiv 3 \bmod 9
$$
and $K_k$ is the unique subfield of $H_k$ that has degree $3^k$ over $K$. For 
$K = \Q(\sqrt{-3})$,  $K_k$ is the unique subfield of degree $3^k$ of $H_{k+1}$. 
For other $D \equiv -3 \bmod 9$, we proceed as follows. As in
the discussion preceding Theorem \ref{expbound}, we let $\alpha_{k+1}\in\OK$ be
locally congruent to $\zeta_3$ modulo $3^{k+1}$.
The fixed field $H_{k+1}'$ of the automorphism $\artin{\alpha_{k+1}}{H_{k+1}/H_0}$ 
has a unique subfield of degree $3^k$ over $K$; this field is the field $K_k$ that
we are after.

\begin{example}
The field $K = \Q(\sqrt{-21})$ has class group isomorphic to $(\Z/2\Z)^2$. The 
index 4 subfield of the
ring class field $H_1$ is generated by a root of $x^3-6x-12$, but it is {\it not\/}
part of the anticyclotomic $\Z_3$-extension. 

The index 4 subfield $\widetilde K_1$ 
of the ring class field $H_2$ is obtained by adjoining a root of
$$
x^9 + 12x^6 + 81x^5 + 144x^4 + 30x^3 - 324x^2 - 504x - 336
$$
to $K$. The Galois group $\widetilde K_1/K$ is isomorphic to $(\Z/3\Z)^2$. 
To obtain the first layer, we compute 
that $\alpha_2 = 1+\sqrt{-21}$ is locally congruent to $\zeta_3$ modulo 9.
We take the fixed field of the Artin symbol corresponding
to $\alpha_2$. We find that $K_1$ is generated by a root of
$$
x^3+9x-12
$$
over $K$.
\end{example}

For the general case, we let $H_{0,3}$ be the 3-Hilbert class field of $K$. 
The extension $H_\infty/H_0$ naturally defines a $\Z_3$-extension 
$H_{\infty,3}/H_{0,3}$. The sequence
\begin{equation}\label{splitseq}
1 \rightarrow \Gal(H_{\infty,3}/H_{0,3}) \rightarrow \Gal(H_{\infty,3}/K) \rightarrow \Gal(H_{0,3}/K)
\rightarrow 1
\end{equation}
need not split in general. If it does split, then $H_{0,3}$ is
{\it not\/} contained in the anticyclotomic $\Z_3$-extension and finding the layers
proceeds as before. Determining whether the sequence splits is often easy.
In Section \ref{mirror}, we will give a simple criterion (Theorem \ref{contain})
under which $H_{0,3}$ is not contained in the anticyclotomic $\Z_3$-extension. 
Furthermore, the following examples show that it is computationally very
easy to determine if $H_{0,3}$ lies in the anticyclotomic $\Z_3$-extension
or not.

\begin{example}\label{exref}
Fix $K = \Q(\sqrt{-87})$. The class group of $\OK$ is cyclic of order 6. The 
order $\O_1$ of index 3 has cyclic Picard group of order 18. We may
replace $H_{\infty,3}$ with $H_{1,3}$ in Sequence (\ref{splitseq}) to 
obtain the nonsplit sequence
$$
1 \rightarrow \Z/3\Z \rightarrow \Z/9\Z \rightarrow \Z/3\Z \rightarrow 1.
$$
Hence, the $3$-part of the Hilbert class field of $K$ is the first
layer of the anticyclotomic $\Z_3$-extension. Explicitly, we have 
$$
K_1 = K[x]/(x^3-x^2+2x+1).
$$
The index 2 subfield of the ring class field for $\O_1$ gives the {\it second\/}
layer of the anticyclotomic $\Z_3$ extension. It is generated by a root of
$$
x^9 + 3x^8 + 6x^7 + 14x^6 + 9x^5 + 21x^4 + 6x^3 + 12x^2 + 3.
$$

For $K = \Q(\sqrt{-771})$ we obtain the split sequence
$$
1 \rightarrow \Z/3\Z \rightarrow \Z/3\Z \times \Z/3\Z \rightarrow \Z/3\Z \rightarrow 1
$$
and the Hilbert class field is not contained in the anticyclotomic $\Z_3$-extension.
\end{example}

If the 3-part of $\Pic(\OK)$ is different from $\Z/3\Z$, the situation is slightly
more involved. In the remainder of this section we explain how to 
split the 3-part $\Pic(\OK)_3$ into a part ``inside'' and a part ``outside''
of the anticyclotomic $\Z_3$-extension.

We let $\Smax\subseteq \Pic(\OK)_3$ be the largest subgroup (with respect to
inclusion) for which the sequence
\begin{equation*}
1 \rightarrow \Gal(H_{\infty,3}/H_{0,3}) \rightarrow \Gal(H_{\infty,3}/H_{0,3}^{\Smax}) \rightarrow \Smax
\rightarrow 1
\end{equation*}
splits. Here, $H_{0,3}^{\Smax}$ is the fixed field of $H_{0,3}$ for $\Smax$.
This fixed field is the
largest subfield of $H_{0,3}$ that is contained in the anticyclotomic $\Z_3$-extension.

For ease of notation, we restrict to the case $D \equiv 3 \bmod 9$
so $H_{\infty,3}$ is the inverse limit of the 3-parts $\Pic(\O_N)_3$ of the
ring class field for $\O_N$.
Let $\langle \gp \rangle \subset \Pic(\OK)_3$
be a subgroup of 3-power order with $\gp$ coprime to 3. The
ideal $\gp \cap \O_N$ is an invertible $\O_N$-ideal whose class in $\Pic(\O_N)_3$
maps to the class of $\gp$ in $\Pic(\OK)_3$. The other preimages are
$(\gp \cap \O_N) I$, with $I$ ranging over the kernel
of $\Pic(\O_N)_3 \rightarrow \Pic(\OK)_3$.
We compute the order inside $\Pic(\O_N)_3$ for each of the preimages
of $\gp$, and check if one of those equals the order
of $[\gp] \in \Pic(\OK)_3$. If it does, the sequence
\begin{equation*}
1 \rightarrow \Gal(H_{N,3}/H_{0,3}) \rightarrow \Gal(H_{N,3}/H_{0,3}^{\langle\gp\rangle})  \rightarrow \langle \gp\rangle \rightarrow 1
\end{equation*}
splits; otherwise it does not.

\begin{example}
Fix $K = \Q(\sqrt{-6789})$. We have $\Pic(\OK) \cong \Z/2\Z \times \Z/6\Z \times \Z/6\Z$
and $\Pic(\O_1) \cong \Z/2\Z \times \Z/6\Z \times \Z/18\Z$. The kernel of the map
$\Pic(\O_1) \rightarrow \Pic(\OK)$ is generated by the class of the $\O_1$-ideal
$$
I = \O_1(9, 3\sqrt{-6789}-3)
$$
of norm 9. There are four subgroups of $\Pic(\OK)$ of index 3; elements of
order 6 in these subgroups are ideals of norm $5$, $7$, $11$ and $97$, respectively.
The ideal $\gp_5$ has order $6$ in $\Pic(\OK)$, but
$I^k (\gp\cap\O_1)$ has order $18$ for $k=0,1,2$. Likewise for $\gp_7$
and $\gp_{97}$. On the other hand, the ideal $(\gp_{11} \cap \O_1)$ has order 6.

The fixed field of $H_0$ under the subgroup
of $\Pic(\OK)$ generated by $\gp_{11}$ (of order 6), $\gp_{5}^3$
(of order 2) and $\gp_2 = \O_1 (2,3\sqrt{-6789}+1)$ (of order 2) equals
the first layer $K_1$ of the anticyclotomic $\Z_3$-extension of $K$. To find a 
generating polynomial, we compute the maximal real subfield of $H_0$ using
CM theory and compute its 4 degree 3 subfields $L_1,\ldots,L_4$. 
We now check whether the Artin symbol corresponding to $\gp_{11}$
acts trivially on $KL_i/K$. As expected, it does so for a unique field. In the
end, we find that a root of 
$$
x^3-x^2+8x+124
$$
generates $K_1/K$.

\end{example}

\section{Practical improvements}\label{shim}

The techniques described yield generating polynomials
that are much larger than necessary. The reason for this is that the $j$-function is
{\it not\/} the right function to use from a practical perspective to compute 
a ring class field. For every given discriminant, a suitably chosen {\it class
invariant\/} can be used instead. The use of class invariants dates back to Weber's
days, and modern treatments rely on {\it Shimura reciprocity\/}. We refer 
to \cite{psh,schertz} for good descriptions and give the main result that we need.

\begin{theorem} Let $D<0$ be a discriminant, and choose a quadratic
generator $\tau$ for the imaginary order of discriminant $D$. Then there 
exists a modular function $f$ of level $n>1$ such that $f(\tau)$ generates the ring
class field; furthermore, the minimal polynomial of $f(\tau)$ over $K = \Q(\sqrt{D})$
can be explicitly computed in time $\widetilde O(|D|)$.
\end{theorem} 
\begin{proof} We refer to \cite[Thm. 4]{schertz} and \cite[Cor. 3.1]{engeeta} for 
two classes of functions. 
\end{proof}

The size of the generating polynomial for the ring class field depends on the choice
of the function $f$ in the theorem. To compute the `reduction factor', we let
$\Psi(j,f)=0$ be the irreducible polynomial relation between $j$ and $f$ and put
$$
r(f) = \frac{\deg_f(\Psi(f,j))}{\deg_j(\Psi(f,j))} \in \Q_{>0}.
$$
As in \cite[Sec. 4]{brpsh}, we expect the logarithmic height of the coefficients
of the minimal polynomial of $f(\tau)$ to be a factor $r(f)$ smaller than
the corresponding coefficients for $j(\tau)$. By \cite[Thm. 4.1]{brpsh}, we have
$$
r(f) \leq 800/7 \approx 114.28.
$$

If $2$ splits in $\O$, then the cube of the Weber-$\gf$ can be used. This
function satisfies $(\gf^{24}-16)^3-j\gf^{24}=0$ and has reduction factor $72/3=24$.
If $2$ is inert, we can use a suitably chosen {\it double $\eta$-quotient\/}. The
exact reduction factor depends on the choice of the $\eta$-quotient; we refer to
\cite{engeeta} for details.
We can use the CM software package \cite{AECM} by Enge to compute the necessary ring class
fields. This package can select the modular function, so that only the discriminant
$D$ is required.

\begin{example}
Let $K = \Q(\sqrt{-3})$. To obtain the first nontrival layer of the 
anticyclotomic $\Z_3$-extension, we compute the ring class field for the order $\O_2$. 
If we use the $j$-function, we obtain a cubic polynomial with constant term
$$
2^{45}\cdot  3\cdot 5^9\cdot 11^3\cdot 23^3.
$$
In this case, a suitably chosen {\it double $\eta$-quotient\/} yields
a class invariant. Using the package \cite{AECM}, we obtain the polynomial
$$
x^3-12x^2-6x-1.
$$
\end{example}

We stress that by class invariants, we can only gain a {\it constant factor\/} in the
size of the coefficients, and that our method is inherently exponential in $\log |D|$.
To push the range of examples further, we can employ {\it lattice basis reduction\/}. 
Indeed, if we have computed a polynomial $f(x)$ that generates the ring class
field, we can view the order defined by $f$ as a {\it lattice\/} in Euclidean space.
If the degree and the coefficients of $f$ are not too big, we can compute a short
basis for this lattice and obtain a ``better'' polynomial.

\begin{example}
For $K = \Q(\sqrt{-3})$, the polynomial $f \in \Z[x]$ for $\O_{3}$ given by Enge's program 
has coefficients in between $-24930$ and $29559$. We view $\Z[x]/(f)$ as a lattice
and after lattice basis reduction, we obtain the
polynomial
$$
x^9 + 9x^6 + 27x^3 + 3.
$$
Using the same technique, we find the polynomial
$$
\eqalign{
x^{27} + 27x^{24} + 324x^{21} + 1980x^{18} + 5022x^{15}\cr
 - 8262x^{12} - 30348x^9 + 304236x^6 + 1365417x^3 + 3}
$$
for the third layer of the anticyclotomic $\Z_3$-extension.
\end{example}

\section{Mirror principle}\label{mirror}

In this section we give an application of the {\it mirror principle\/} that 
relates the class groups of the imaginary quadratic field $\Q(\sqrt{D})$ and 
the real quadratic field $\Q(\sqrt{-D/3})$. This allows us to prove the
following theorem that was alluded to in Example \ref{exref}.

\begin{theorem}\label{contain} Let $D \equiv 3 \bmod 9$ be a negative
discriminant, and 
assume that $3$ does not divide the class number of the real quadratic field
$\Q(\sqrt{-D/3})$. Then the $3$-Hilbert class field of $K = \Q(\sqrt{D})$ is 
contained in the anticyclotomic $\Z_3$-extension of $K$.
\end{theorem}

The proof of the theorem relies on the following lemma. The proof of this
lemma is very similar to the proof of Scholz' mirror theorem \cite{Sch}.

\begin{lemma} \label{unique} Let $D \equiv 3 \bmod 9$ be a negative
discriminant, and assume that 3 does not
divide the class number of the real quadratic field $\Q(\sqrt{-D/3})$. Then, 
there 
exists exactly one degree 3 extension of $\Q(\sqrt{D})$ that is unramified
outside 3 and dihedral over $\Q$.
\end{lemma}
\begin{proof}
Let $K = \Q(\sqrt{D})$ and let $L/K$ be a degree 3 extension that is
unramified outside 3 and dihedral over $\Q$.
The field $L$ fits inside Diagram \ref{uniquediag} below. This diagram also defines 
automorphisms $\tau, \sigma$ and $\varphi$. By abuse of notation, $\tau$ denotes
both a generator of $\Gal(L/K)$ and its unique lift to $\Gal(L(\zeta_3)/V)$; likewise
for $\sigma$ and $\varphi$.
Because $L(\zeta_3)/V$ is
a Kummer extension, we can write $L(\zeta_3) = V(\sqrt[3]\alpha)$ with $\alpha \in V$.  

\begin{figure}
\renewcommand{\figurename}{Diagram}
\begin{center}
\begin{tikzcd}
& L(\zeta_3) \arrow[dl, dash] \arrow[dr, , dash, "\tau"] \\
L \arrow[dr, dash, "\tau"] & & V = K(\zeta_3) \arrow[dl, dash, "\varphi"] \arrow[dr, dash, "\sigma"]\\
& K \arrow[dr, dash, "\sigma"] & \Q(\zeta_3) \arrow[u, dash] & F =\Q(\sqrt{-D/3}) \\
& & \Q \arrow[u, dash] \arrow[ur, dash, "\varphi"] 
\end{tikzcd}
\end{center}
\caption{Diagram of fields for Lemma \ref{unique}}.\label{uniquediag}
\end{figure}

Any such $L(\zeta_3)$ will have $\varphi$ acting trivially on the 
corresponding $\tau$ as well as have $\sigma$ acting as $-1$ on $\tau$.
Our proof proceeds by showing that both the field of definition and the
norm of $\alpha$ are very restricted.

First we show that
$\alpha$ can be taken to lie in the real quadratic field $F = \Q(\sqrt{-D/3})$. 
The {\it Kummer pairing\/}
$$
\langle \alpha \rangle / \langle\alpha^3\rangle \times \langle \tau \rangle \rightarrow \mu_3
$$
is Galois equivariant, and 
since $\sigma$ acts on $\zeta_3$ as $-1$ and on $\tau$ as $-1$, we see that $\sigma$
acts as $+1$ on $\alpha\bmod (V^*)^3$. We deduce that $\sigma(\alpha) = \alpha \cdot \beta^3$ for some $\beta\in V^*$,
and hence
$$
N_{V/F}(\alpha) = \alpha \sigma(\alpha) \equiv \alpha^2 \bmod \hbox{cubes}.
$$
Since $\alpha$ and $\alpha^2$ generate the same extension, this 
shows that we may assume that $\alpha$
lies in $F$.

Since the extension $L(\zeta_3)/V$ is unramified outside $3$, we have $(\alpha) = I J^3$
for ideals $I,J$ 
with $I$ a product of primes lying over $(3) \subset \Z$. The assumption
that $3$ does not divide the class number of $F$ now implies that
we may assume that $\alpha$ is $3$-unit.
Furthermore, the assumption $D \equiv 3 \bmod 9$ implies
that $3$ is inert in $\Q(\sqrt{-D/3})$.  
We get that $\alpha$ is a unit times $3^a$ for some $a$. Since
$\varphi(\alpha)$ is congruent to $\alpha^{-1}$ modulo cubes, we must have 
$a\equiv 0 \bmod 3$.  Therefore, we may take
$\alpha =\pm \varepsilon$, with $\varepsilon$ a fundamental unit of $F$.  
\end{proof}


\begin{proof}[Proof of Theorem \ref{contain}]
Since $K$ has a unique cubic extension that is unramified outside 3 and
dihedral over $\Q$, the class group of $\OK$ has 3-rank at most 1. We write
$\Pic(\OK)_3 = \Z/3^n\Z$ for some $n \geq 0$. We need to prove that the 3-Hilbert class
field $H_3(K)$ coincides with the $n$-th level $K_n$. 

Suppose that we have $H_3(K) \cap K_n = K_k$ for some $k < n$. The Galois group of the 
compositum $H_3(K) K_n$ over $K$ then has 3-rank 2. This means that there is more than
one cubic extension of $K$ contained in $H_3(K) K_n$. All these extensions are unramified
outside 3 and dihedral over $\Q$ however; contradiction. 
\end{proof}

Lemma \ref{unique} allows us to deduce a simple sufficient criterion for when the
3-parts $\Pic(\O_N)_3$ are cyclic.

\begin{theorem} \label{cyclicPic} Assume that $3$ does not divide the
class number of the real quadratic field $\Q(\sqrt{-D/3})$. 
For $D \equiv 3 \bmod 9$, the $3$-part $\Pic(\O_N)_3$ 
is cyclic for all $N \geq 0$. 
\end{theorem}
\begin{proof}
By Theorem \ref{contain}, the sequence 
$$
1 \rightarrow (\OK/3^N\OK)^*/\image(\OK^*)(\Z/3^N\Z)^* \rightarrow \Pic(\O_N)_3 \rightarrow \Pic(\OK)_3 \rightarrow 1
$$
does not split for any $N$. Since the first and last term are cyclic, this means that
the middle term is cyclic.
\end{proof}


%

\section{Generators via Kummer theory}\label{kummersec}

In computational class field theory, the `standard' way to compute an abelian
extension of prescribed conductor of a number field $K$ depends on whether $K$
has the appropriate roots of unity. If it does, we can use Kummer theory. If it
does not, we adjoin the right root of unity $\zeta_n$ to $K$ and compute the
right abelian extension of $K(\zeta_n)$ first. Afterwards, we `descend' down 
to $K$. We refer to \cite{cohenpsh} for a detailed description. 

If $K$ is imaginary quadratic, we can use complex multiplication techniques
instead and bypass the general method. This is the technique we used in
Section~\ref{ringclassfield}. 
However, we can make the Kummer theory approach very explicit in our setting. As
before, $K = \Q(\sqrt{D})$ is an imaginary quadratic field in which $3$ ramifies.
Throughout this section, we assume 3 does not divide the class number
of the real quadratic field $F = \Q(\sqrt{-D/3})$; we also assume that $3$
remains inert in $F$. This last restriction is essential in Lemma \ref{gpring}; the
split case appears to be much harder.

\begin{theorem}\label{kummerthm}
Assume that $3$ ramifies in $K=\Q(\sqrt{D})$ and that $3$ is inert in
$F = \Q(\sqrt{-D/3})$. If, furthermore, 3 does not divide the class
number of $F$, then the expression for $\kappa_n$ given in Definition \ref{kappagen} 
gives a Kummer generator for $K_n(\zeta_{3^n})/K(\zeta_{3^n})$ for $n \geq 1$.
\end{theorem}

Once we have computed $\kappa_n$, we can use the technique from
\cite[pp.\ 514--515]{cohenpsh} to descend from $K_n(\zeta_{3^n})/K(\zeta_{3^n})$ down to $K_n/K$.

Diagram \ref{kummerdiag} defines the various fields we will work in and explains the 
inclusion relations between them. In this diagram, the $+$ notation indicates the
maximal real subfield. We write $d = -D/3$, so that $F=\Q$
for $D=-3$ and $F$ is real quadratic otherwise. If $F$ is quadratic,
we let $\chi$ be the associated quadratic character of conductor $d$. 

All the base fields we
consider are subfields of $L = \Q(\zeta_{3^nd})$; we identify $\Gal(L/\Q)$ with 
$(\Z/3^nd\Z)^*$ and for an integer $b$ with $\gcd(b,3d) = 1$, we let
$\sigma_b$ be the automorphism satisfying $\sigma_b(\zeta_{3^nd}) =
\zeta_{3^nd}^b$. For $d\not = 1$, we fix an integer $c \equiv -1 \bmod 3^n$
with $\chi(c) = -1$; we identify $\Gal(F_n/T_n) \cong \Gal(F/\Q) \cong
\langle \sigma_c \rangle$ in this case. (For $d=1$, all statements about
$\sigma_c$ play no role and should be ignored.)

\begin{lemma}\label{coprime} The class number of $F_n$ is coprime to 3.
\end{lemma}
\begin{proof} By assumption, $3$ remains inert in $F/\Q$. As
the extension $F_n/F$ has only one ramified prime and is totally
ramified, the lemma follows from \cite[Thm.~10.4]{lcw}.
\end{proof}


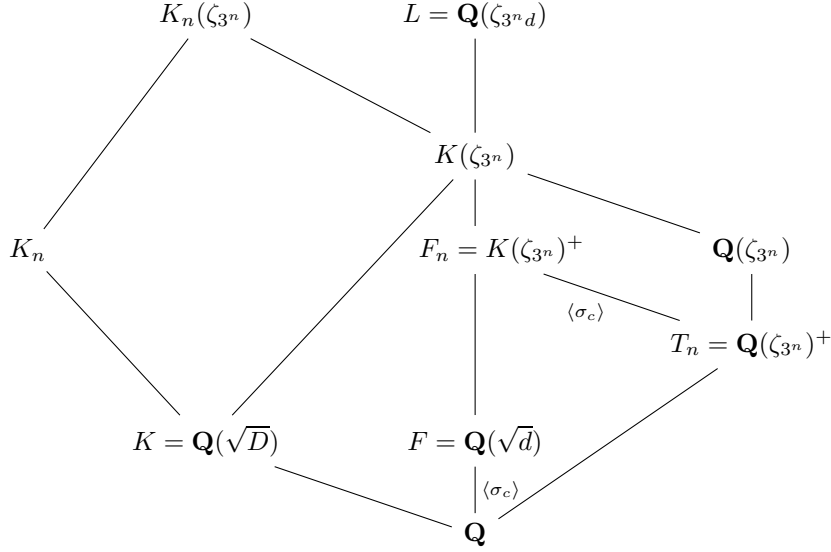
\begin{figure}
\renewcommand{\figurename}{Diagram}
\begin{center}
\begin{tikzcd}
&K_n(\zeta_{3^n}) \arrow[rdd,dash] & L=\Q(\zeta_{3^n d}) \arrow[dd,dash]\\
\\
&&K(\zeta_{3^n}) \arrow[d,dash] \arrow[dr,dash] \\
K_n \arrow[ruuu,dash]&& \qquad F_n = K(\zeta_{3^n})^+ \arrow[dd,dash] & \Q(\zeta_{3^n}) \arrow[d,dash] \\
&& & T_n = \Q(\zeta_{3^n})^+ \arrow[ddl,dash] \arrow[ul, dash, "\langle \sigma_c\rangle "] \\
&K=\Q(\sqrt{D}) \arrow[dr,dash] \arrow[luu,dash] \arrow[ruuu, dash] & F = \Q(\sqrt{d})\arrow[d,dash, "\langle \sigma_c \rangle"] \\
&& \Q
\end{tikzcd}
\end{center}
\caption{Diagram of fields for Section \ref{kummersec}}.\label{kummerdiag}
\end{figure}

The techniques of the proof of Lemma \ref{unique} shows that we may assume that the desired
element $\kappa_n$ lies in $F_n$. Furthermore, since the class number
of $F_n$ is coprime to $3$, this proof also shows that $\kappa_n$ is a
3-unit in $F_n$.


Furthermore, we claim that we may assume 
that $\sigma_c$ inverts $\kappa_n$. To see this, note that
$K(\zeta_{3^n d})/K$ is disjoint from $K_{\infty}/K$,  and $\sigma_c$ therefore 
acts trivially on $\Gal(K_n(\zeta_{3^n})/K(\zeta_{3^n}))$. 
It also acts by inversion on $\zeta_{3^n}$. Therefore, the Kummer pairing
tells us that $\sigma_c$ acts by inversion on $\kappa_n$ 
modulo $3^n$-th powers, i.e.,
we have $\kappa_n^{\sigma_c}=\kappa_n^{-1}\gamma^{3^n}$ for some $\gamma$. But 
then $\kappa_n^{1-\sigma_c} = \kappa_n^2\gamma^{-3^n}$ generates the same 
extension and is inverted by $\sigma_c$ since $\sigma_c^2=1$ 
on $K(\zeta_{3^n})$. 

Let $E_n$ be the group of 3-units of $F_n$. Let $E_n^-$ denote the subgroup consisting 
of elements that are inverted by $\sigma_c$, and $E_n^+$ denote those 
that are fixed (and hence lie in $T_n$).
We will compute a valid $\kappa_n$ as a product of suitably chosen
3-units in $E_n^-$.
For $n\ge 1$, we define
$$
\xi_n= \prod_{\substack{1\le a \le 3^nd \\  a\equiv \pm 1\text{ mod }{3^n}\\ 
(a, d)=1}} (1-\zeta_{3^{n}d}^a)^{\chi(a)}.
$$
The product is over values of $a$ representing elements of $\Gal(L/T_n)$. We
claim that $\xi_n$ lies in $F_n$. Indeed, for $\sigma_b \in \Gal(L/F_n)$ we
have $b \equiv \pm 1 \bmod 3^n$ and $\chi(b) = 1$. The computation
\begin{align*}
\sigma_b(\xi_n) &= \prod_{\substack{1\le a \le 3^nd \\  a\equiv \pm 1\text{ mod }{3^n} \\ (a, d)=1}} (1-\zeta_{3^{n}d}^{ab})^{\chi(a)}
&= \prod_{\substack{1\le a \le 3^nd \\  a\equiv \pm 1\text{ mod }{3^n} \\ (a, d)=1}} (1-\zeta_{3^{n}d}^a)^{\chi(a/b)} = \xi_n
\end{align*}
gives $\xi_n \in F_n$. 

\begin{lemma}\label{xiprop}
\begin{enumerate}
\def\theenumi{\alph{enumi}}
\item $\xi_n\in E_n^-$.
\item The norm of $\xi_n$ from $F_n$ to $F_{n-1}$ is $\xi_{n-1}$.
\end{enumerate}
\end{lemma}
\begin{proof} For part (a), a simple computation shows that $\sigma_c(\xi_n)
= \xi_n^{-1}$.  If $d\ne 1$, every factor $1-\zeta_{3^n d}^a$ is a unit, 
so $\xi_n$ is a unit. If $d=1$, then each factor is 
a 3-unit. Therefore, $\xi_n\in E_n^-$.

For (b), we note that the Galois conjugates of $\zeta_{3^n}$ 
for $L/\Q(\zeta_{3^{n-1} d})$ are $\zeta_{3^n}\zeta_3^i$ for $i=0,1,2$.
Therefore, the norm of the factor $(1-\zeta_{3^{n} d}^a)$ is 
$$
\prod_{i=0,1,2} (1-\zeta_{3^{n} d}^a\zeta_3^{i}) = (1-\zeta_{3^{n} d}^{3a}) 
= (1-\zeta_{3^{n-1} d}^a),
$$
and the result folows.
\end{proof}

\begin{lemma}\label{gpring}
The $\sigma_j(\xi_n)$ for $\sigma_j\in\text{Gal}(F_n/F)$ are independent 3-units and generate a subgroup of $E_n^-$ of index 
prime to 3.
\end{lemma}
\begin{proof}
We need some preliminary work.
Since keeping track of powers of 2 is irrelevant for what we do, for numbers $a$ and $b$ we use the notation 
$a\approx b$ to say that $a/b$ is a power of 2, up to sign.
When $a, b$ are groups, $a\approx b$ means that $a$ and $b$ are subgroups of some larger 
group $G$ with $[G:a]/[G:b]$ equal to a power of 2.

Since $\sigma_c^2=1$ on $F_n$, the identity $x^2= x^{1-\sigma_c}x^{1+\sigma_c}$, implies 
$$
E_n\approx E_n^- \oplus E_n^+.
$$ 
Let $\{u_1, \dots, u_{3^{n-1}}\}$ be a basis for
$E_n^-$ and $\{v_1, \dots v_{3^{n-1}-1}\}$ be a basis for $E_n^+$ mod $\{\pm 1\}$. The Galois group 
of $F_n/\Q$ 
is given by the elements
$\sigma_j$ and $\sigma_c\sigma_j$, where $\sigma_j$ runs through $\text{Gal}(F_n/F)$. 
These can be used to calculate the regulator $R_n$ of $F_n$, up to
powers of 2. Let
$$
R_n^-=(\log|\sigma_j(u_i)|_{j,i}) \qandq R_n^+=(\log|\sigma_j(v_i)|_{j,i}).
$$
Then $R_n$, up to powers of 2, is the absolute value of the determinant of the matrix
$$
\begin{pmatrix} R_n^- & R_n^+\\ -R_n^- & R_n^+ \end{pmatrix}
$$
with the last row deleted. Adding the top rows to the corresponding bottom rows 
yields a 0-block in the lower left and twice $R_n^+$ in the lower right.
Therefore, 
$$
R_n\approx \det(R_n^-) \det(R_n^+).
$$
Note that $\det(R_n^+)$ is, up to powers of 2, the regulator of $T_n$.

We define the regulator
$$
R_{\xi_n} = \abs{\det(\log|\sigma_j\sigma_i^{-1} \xi_n|)}, 
\qquad i,j \in \Gal(F_n/F)
$$
of the $\Gal(F_n/F)$-conjugates of $\xi_n$. We claim that
$$
\det(R_n^-)\approx \frac{h(F_n)R_{\xi_n}}{h(T_n)}.
$$
holds. (Here, $h(\cdot)$ denotes the class number.) 
Since $R_{\xi_n}/\det(R_n^-)$ is the index in $E_n^-$ of the 
subgroup generated by the conjugates of $\xi_n$, the lemma then
follows from the observation that
both $T_n$ and $F_n$ have class number coprime to 3.

The value $R_{\xi_n}$ is a group determinant, and 
by \cite[Lemma~5.26]{lcw} we have
$$
R_{\xi_n} =\pm \prod_{\psi} \sum_{j} \psi(\sigma_j)\log|\sigma_j \xi_n|,
$$
where $\psi$ ranges over the Dirichlet characters 
for $\Gal(F_n/F)\simeq \Gal(T_n/\Q)$,
and $\sigma_j$ ranges over $\Gal(F_n/F)$. 

We have
$$
\sum_{j} \psi(j)\log|\sigma_j \xi_n|= \sum_{j} \psi(j)\sum_{a}\chi(a) \log|1-\zeta_{3^n d}^{aj}|,
$$
where $1\le a \le 3^n d$, $(a,d)=1$, $a\equiv \pm 1\bmod{3^n}$.
This equals
$$
\sum_{1\le a\le 3^n d,\, (a,3d)=1} \psi(a)\chi(a) \log|1-\zeta_{3^n d}^{a}|.
$$

Recall that if $\psi$ has conductor $3^m$ with $m\ge 1$, then
$$
L(1, \overline{\psi\chi}) = -\frac{g(\overline{\psi\chi})}{3^m d} \sum_{1\le a \le 3^m d,\,  (a,3d)=1}
 \psi(a)\chi(a) \log|1-\zeta_{3^m d}^a|,
$$
where $g(\overline{\psi\chi})$ is a Gau\ss\ sum.
Since the values of $\psi(a)$ depend only on $a \bmod 3^m$, we have, for fixed $a_0$ with $3\nmid a_0$,
$$
\sum_{\substack{1\le a\le 3^n d\\  a\equiv a_0\bmod 3^m d}} \psi(a) \log|1-\zeta_{3^n d}^a| = 
\psi(a_0) \log|1-\zeta_{3^m d}^{a_0}|,
$$
where we have used the identity $\prod_{\omega^{3^{n-m}}=1} (1-\omega x) = 1-x^{3^{n-m}}$.
Therefore, 
$$
\sum_{j} \psi(j)\log|\sigma_j \xi_n| = \frac{3^m d}{g(\overline{\psi\chi})} L(1, \overline{\psi\chi}).
$$

If $\psi$ is trivial, then 
\begin{align*}
\sum_{j}\log|\sigma_j \xi_n|&= \log|\text{Norm}_{F_n/F_1}\xi_n|=\log|\xi_1|\\
&=\sum_{1\le a\le 3d\atop (a,3d)=1} \chi(a) \log|1-\zeta_{3d}^a|.
\end{align*}
For fixed $a_0$, 
$$
\sum_{\substack{a\equiv a_0\bmod d \\ 1\le a \le 3d, (a,3d)=1}} \log|1-\zeta_{3d}^a|
=\log|1-\zeta_{3d}^{3a_0}| - \log|1-\zeta_{3d}^{3a_1}|,
$$
where $3a_1\equiv a_0\pmod d$ and $1\le 3a_1\le 3d$. Therefore,
$$
\eqalign{
\sum_{j}\log|\sigma_j \xi_n|&= \sum_{1\le a_0\le d\atop (a_0, d)=1} \chi(a_0) \log|1-\zeta_d^{a_0}| - \sum_{1\le a_1\le d\atop (a_1, d)=1} \chi(3a_1) \log|1-\zeta_d^{a_1}|\cr
&=(1-\chi(3) )\frac{-d}{g(\chi)}L(1, \chi).
}
$$
Using that $3$ is inert in $F/\Q$, we compute $1-\chi(3)=2$. 

From the analytic class number formula, we derive
$$
\frac{h(F_n)R_n}{\sqrt{\disc(F_n)}}\, \frac{\sqrt{\disc(R_n)}}{h(T_n)R_n^+}\approx  \prod_{\psi} L(1, \overline{\psi\chi}),
$$
where $h$ denotes the class number of the indicated field.
By \cite[Thm.\ 3.11 and Cor.\ 4.6]{lcw}, the Gau\ss\ sums, the discriminants, and the 
conductor $3^m d$ factors cancel, and we obtain
\begin{equation*}
\det(R_n^-)\approx \frac{h(F_n)R_{\xi_n}}{h(T_n)}.
\qedhere
\end{equation*}
\end{proof}

As a byproduct of the calculation with $\psi=1$, we obtain the 
following:
\begin{lemma}\label{basecase} If $d\ne 1$, then $\xi_1=\epsilon_0^{-4h(F)}$, 
where $h(F)$ and $\epsilon_0$ are the class number 
and fundamental unit of $F$. If $d=1$ then $\xi_1=3$.
\end{lemma}
\begin{proof} Up to sign, the case $d\ne 1$ results from keeping track of the factors of 2. 
In the definition of $\xi_1$, we can pair the factors for $a$ and $3d-a$ to see 
that $\xi_1$ is totally positive.  When $d=1$, the result follows directly 
from the definition of $\xi_1$. 
\end{proof}

We have almost done all the preparatory work to construct $\kappa_n$. Indeed,
by Lemma \ref{gpring} we know that $\kappa_n$ is a product of Galois
conjugates of $\xi_n$. To pin down the product, we need the following
standard result.

\begin{lemma}\label{Kummer} Let $m\ge 1$.
Let $M$ be a number field, let $\zeta_{m}$ be a primitive $m$-th root of 
unity, and let $\alpha\in M(\zeta_{m})^{\times}$.
Let $M(\zeta_{m},\alpha^{1/m})/M(\zeta_{m})$
be a cyclic extension of degree $m$. Define a map
$$ 
\omega: \, \Gal(M(\zeta_{m})/M)\to \Z/m\Z
$$
by $\tau(\zeta_m) = \zeta_m^{\omega(\tau)}$. Then
$F(\zeta_{m},\alpha^{1/m})/M$ is Galois with 
abelian Galois group if and only if 
$$ 
\alpha^{\tau-\omega(\tau)} \in \bigl(M(\zeta_{m})^{\times}\bigr)^{m}
$$
holds for all $\tau\in \Gal(M(\zeta_m)/M)$.
\end{lemma}
\begin{proof} The proof is a standard calculation with the Kummer pairing. 
See for instance the proof of \cite[Thm. 14.7]{lcw}. 
\end{proof}

Choose $\tau\in \Gal(F_n/F)$ satisfying $\tau(\zeta_{3^n})=\zeta_{3^n}^4$. 
We have
$$
\kappa_n=\prod_{j=0}^{3^{n-1}-1} \tau^j(\xi_{n})^{c_j},
$$
for some integers $c_j$. Therefore, taking indices mod $3^{n-1}$, we have
$$
\kappa_n^{\tau} = \prod_{j=1}^{3^{n-1}} \tau^{j}(\xi_{n})^{c_{j-1}}.
$$
Lemma \ref{Kummer} says that $\kappa_n^{\tau-4}$ is a $3^n$-th power, and 
since the elements $\tau^j(\xi_{n})$ are multiplicatively independent,
we must have
$$
c_{j-1}-4c_{j}\equiv 0\pmod{3^n}, \qquad 1\le j\le 3^{n-1}.
$$
This implies that each $c_j$ is uniquely determined mod $3^n$ by the 
value of $c_0$. Therefore, $\kappa_n$
is uniquely determined up to an integral power and mod $3^n$-th powers. 
Therefore, if we find $\kappa_n\in F_n$ such that
\begin{enumerate}
\item $\kappa_n^{\tau-4}$ is a $3^n$th power
\smallbreak
\item $\kappa_n$ is not a cube in $F_n$,
\end{enumerate} 
then we have a Kummer generator for $K_n(\zeta_{3^n})/K(\zeta_{3^n})$.

For $i\ge 1$, define
$$
B_i=\prod_{j=1}^{i-1}\Biggl(\frac{1+4^{3^{j-1}}+16^{3^{j-1}}}{3}\Biggr).
$$
Let
$$
b_i=(1-B_i)/3,
$$
which is an integer for all $i\ge 1$.
Finally, for $i\ge 2$, let
$$
D_i(x)=\frac{3(b_i(x-1)-1)-(1+x^{3^{i-1}}+x^{2\cdot 3^{i-1}})(b_{i-1}(x-1)-1)}{x-4}.
$$
Note that the numerator of $D_i(x)$ evaluated at $x=4$ is
$$
3(3b_i-1)-(1+4^{3^{i-1}}+16^{3^{i-1}})(3b_{i-1}-1)
= 3(-B_i)+(1+4^{3^{i-1}}+16^{3^{i-1}})B_{i-1}=0,
$$
so $D_i$ has integer coefficients. For example, $D_2(x)=x-1$.

Let
$$
\delta_i=\xi_i^{D_i(\tau)} \text{ for } i\ge 2, \quad 
\beta_i=\xi_i^{b_i(\tau-1)-1} \text{ for } i\ge 1.
$$
Then $\xi_i, \beta_i, \delta_i\in F_i$, and
$$
\delta_i^{\tau-4} = \frac{\beta_i^3}{\beta_{i-1}}
$$
for $i\ge 2$.  Moreover,
$$
\beta_1=\xi_1^{b_1(\tau-1)-1}=\xi_1^{-1}.
$$

\begin{definition}\label{kappagen}
Let  $\kappa_1=\xi_1$, and for $n\ge 2$ let
$$
\kappa_n=\xi_1\delta_2^3\cdots \delta_{n}^{3^{n-1}}\in F_n\subset K(\zeta_{3^{n}}).
$$
\end{definition}

We have
\begin{align*}
\kappa_n^{\tau-4}&=\xi_1^{-3}\frac{\beta_2^9}{\beta_1^3}\frac{\beta_3^{27}}{\beta_2^9} \cdots 
\frac{\beta_{n}^{3^n}}{\beta_{n-1}^{3^{n-1}}}\\
&= \beta_{n}^{3^n}.
\end{align*}

\begin{lemma}\label{final}
$\kappa_n$ is not a cube in $K(\zeta_{3^n})$.
\end{lemma}
\begin{proof}
\meshing
The lemma is equivalent to $\xi_1$ not being a cube in $\Q(\zeta_{3^n})$. 
Our assumption $3\nmid h(F)$
implies that
$\xi_1=\epsilon_0^{-4h(F)}$ is not a cube in $F$ (when $d\ne 1$; 
the case $d=1$ is trivial), so $x^3-\xi_1$ generates a non-Galois cubic 
extension of $F$ that must be disjoint from every abelian extension.
Therefore, $\sqrt[3]{\xi_1}\not\in  K(\zeta_{3^n})$.
\end{proof}

\begin{proof}[Proof of Thm.\ \ref{kummerthm}]
Lemma \ref{final} implies that
$$
K(\zeta_{3^n}^{\phantom{1}})(\sqrt[3^n]{\kappa_n}) / K(\zeta_{3^n}^{\phantom{1}})
$$
is cyclic of order $3^n$.
Since $\kappa_n^{\tau-4}$ is a $3^n$-th power and $\kappa_n$ is real, it is the 
desired Kummer generator for $K_n(\zeta_{3^n})/K(\zeta_{3^n})$.
\end{proof}

\begin{example} For $K = \Q(\sqrt{-3})$, we have $\kappa_1=3$ and hence
$K_1 = \Q(\sqrt{-3},3^{1/3})$. 

To obtain the second layer, we compute $D_2(x) = x-1$ and
$$
\kappa_2= 3
\biggl(
\frac{(1-\zeta_9^4)(1-\zeta_9^{-4})}{(1-\zeta_9)(1-\zeta_9^{-1})}
\biggr)^3
 = 
3
\biggl(
\frac{1-\cos(8\pi/9)}{1-\cos(2\pi/9)}
\biggr)^3
.
$$
We compute that $\kappa_2$ is a root of $x^3 - 1710x^2 + 513x - 27$, and
$\kappa_2^{1/9}$ is therefore a root of
$$
x^{27}-1710x^{18}+513x^9-27.
$$
Having found the extension $K(\zeta_9)(\kappa_2^{1/9})/K(\zeta_9)$,
we proceed as in \cite[pp.\ 514--515]{cohenpsh} to descend to the extension
$K_2/K$. We compute that $K_2$ is generated over $K$ by a root of
$$
x^9 - 59049x^3 + 4251528\sqrt{-3}.
$$
\end{example}

\begin{example}
Fix $K = \Q(\sqrt{-87})$. For the first layer, we compute
that $\kappa_1$ is a root of $x^2-727x+1$. Instead of following the descent
procedure from \cite{cohenpsh}, we can also use the following argument to
compute $K_1/K$. We replace $x$ by $x^3$ and take the compositum with
$x^2+3$ to obtain a degree $12$ polynomial 
defining $K(\zeta_9)(\kappa_1^{1/3})/\Q$. This field has 7 subfields of
degree 6. We test these fields pairwise for isomorphism, and compute that
there is a unique field that is not isomorphic to another field. Hence,
this is the unique field that is Galois over $\Q$ and must equal $K_1$. 
Applying lattice basis reduction to the default generator of $K_1/\Q$ gives the 
polynomial
$$
x^6 - 3x^5 + 13x^4 - 21x^3 + 43x^2 - 33x + 9.
$$

To obtain $K_2$, we compute that $\kappa_2$ is a root of
{
\setlength{\jot}{3pt}
\begin{align*}
x^6 - &3298753006106830814034741x^5 + 8591489279598602990016127145116806x^4\\
& - 28320363968461011184065689777889416199793x^3\\
& + 8591489279598602990016127145116806x^2\\
& - 3298753006106830814034741x + 1.
\end{align*}
}\null
The same technique as for $K_1$ gives that there are {\it two\/} subfields
of $K_1(\zeta_9)(\kappa^{1/9})/\Q$ that are Galois over $\Q$. Since one
of them is the known field $K_1T_1$, we select the field $K_2$ to
be the other subfield that is Galois over $\Q$. A generating polynomial
is given in Example \ref{exref}.
\end{example}

\section*{Acknowledgements}
We thank the referees for valuable suggestions on an earlier version of this
paper.

\providecommand{\bysame}{\leavevmode\hbox to3em{\hrulefill}\thinspace}
\providecommand{\MR}{\relax\ifhmode\unskip\space\fi MR }
\providecommand{\MRhref}[2]{%
  \href{http://www.ams.org/mathscinet-getitem?mr=#1}{#2}
}
\providecommand{\href}[2]{#2}

%
%
%
%

\end{document}